\documentclass[11pt]{article}
\usepackage{latexsym}
\usepackage{amssymb}
\usepackage{amscd} 
\usepackage{amsmath} 
\usepackage{amsrefs}
\usepackage{amsthm}
\usepackage{color}
\usepackage{geometry}
\usepackage{graphicx} 
\usepackage[hidelinks]{hyperref}
\hypersetup{
    colorlinks=true,
    allcolors=black,  
}
\usepackage{mathtools}
\usepackage{subcaption}
\usepackage{tikz}
\usepackage{tikz-cd}
\usetikzlibrary{arrows}
\usetikzlibrary{patterns}
\usepackage{enumitem}

\normalsize

\def\R{{\mathbb R}}

\def\del{\partial}

  \newtheorem{theorem}{Theorem}[section]
  \newtheorem{lemma}[theorem]{Lemma}

  \theoremstyle{definition}

\begin{document}
\title{Small 4-regular planar graphs that are not \\ circle representable}

\author{Jane Tan\footnote{Supported by the Australian Research Council, Discovery Project DP140101519.}\\
 \small Mathematical Sciences Institute, Australian National University\\[-0.6ex]
 \small Canberra, ACT 2601, Australia\\[-0.6ex]
 \small \texttt{jane.tan@sjc.ox.ac.uk}
}
\date{}
%University of Oxford, Oxford OX2 6GG, UK
%
\maketitle              
\vspace{-5mm}
\begin{abstract}
A 4-regular planar graph $G$ is said to be circle representable if there exists a collection of circles drawn on the plane such that the touching and crossing points correspond to the vertices of $G$, and the circular arcs between those points correspond to the edges of $G$. Lov\'asz (1970) conjectured that every 4-regular planar graph has a circle representation, but an infinite family of counterexamples was given by Bekos and Raftopoulou (2015). We reduce the order of the smallest known counterexamples among simple graphs from 822 to 68 based on a multigraph counterexample of order 12.
\end{abstract}

\section{Introduction}
A \emph{circle representation} of a 4-regular planar (multi)graph $G$ is a collection of circles embedded in $\R^2$ such that each point of the plane belongs to at most two circles, the set of points belonging to exactly two circles is in bijective correspondence with the vertex set of $G$, and the circular arcs between those points correspond to the edge (multi)set of $G$. A point in the intersection of two circles is a \emph{touching} point if it is the only point at which those circles intersect, and a \emph{crossing} point if it is one of the two points at which they intersect. We say that a graph is \emph{circle representable} if it admits a circle representation. Two examples of circle representable graphs are given in Figure~\ref{fig:exeasy}.

\begin{figure}
	\small
	\centering
	\unitlength=1cm
	 \scalebox{1}{
\begin{picture}(10.3,1.5)(0,0)
  \put(0,0){\includegraphics[scale=0.3]{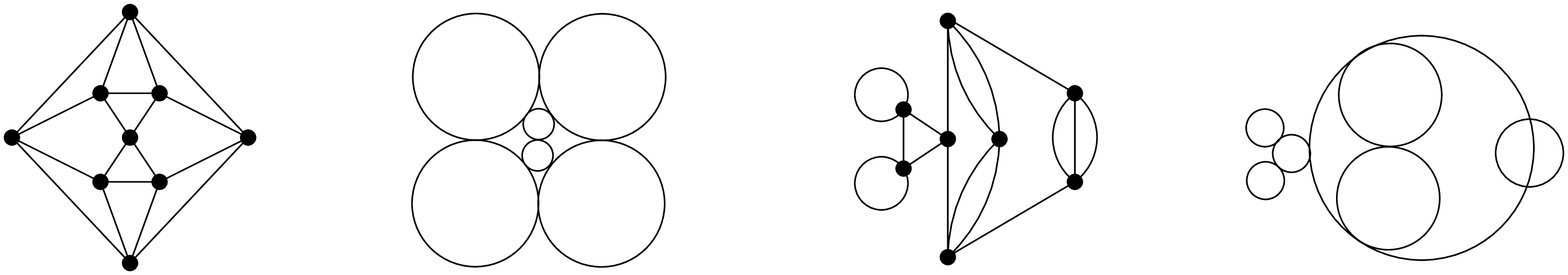}}
\end{picture}}
\caption{Circle representations of a simple graph and a multigraph.}
\label{fig:exeasy}
\end{figure}

Circle representations are closely related to the classical \emph{coin representations} of graphs, which are a collection of interior-disjoint circles in $\R^2$ such that the circles are in bijective correspondence with the vertex set, and two vertices are adjacent in $G$ if and only if their corresponding circles touch. The remarkable fact that all simple planar graphs admit a coin representation, originally proved by Koebe~\cite{koebe}, is known as the Circle Packing Theorem. Further representations of graphs involving circles and disks have since been studied in literature (see, e.g.,~\cite{CCJ90,bs,DGS14,klee}).

In 1970, Lov\'asz~\cite{lovaszq} conjectured an analogue of the Circle Packing Theorem for circle representations: that every simple 4-regular planar graph admits a circle representation. Bekos and Raftopoulou~\cite{br} showed that the conjecture holds for simple 3-connected 4-regular planar graphs as a consequence of the Circle Packing Theorem, and in the general case presented two infinite families of counterexamples, one of which consists of 2-connected graphs. To describe these, we introduce some terminology following~\cite{br}. 

Let a \emph{gadget-subgraph} be the graph shown in Figure~\ref{fig:bigcounter}(b) where the shaded loop may be replaced with any plane graph that is 4-regular except for one vertex of degree 2 on the outer face. We shall call such a graph a \emph{mini-gadget}. The octahedral mini-gadget, obtained by replacing one edge of the octahedron graph with a path of length 2, is the smallest and yields the gadget-subgraph shown in Figure~\ref{fig:bigcounter}(c). Similarly, Figures~\ref{fig:bigcounter}(d) and (e) depict \emph{bigadget-subgraphs} with abstract and octahedral \emph{mini-bigadgets} respectively. The counterexamples in~\cite{br} are then constructed by taking a subdivision of the octahedron shown in Figure~\ref{fig:bigcounter}(a), and attaching one copy of a (bi)gadget-subgraph for each dotted line in the figure by identifying the degree 2 vertices in the gadget to those in the base. Using only bigadget-subgraphs produces a 2-connected counterexample.

\begin{figure}[h]
	\small
	\centering
	\unitlength=1cm
	 \scalebox{1}{
\begin{picture}(11.2,4)(0,0)
  \put(0,0.2){\includegraphics[scale=0.24]{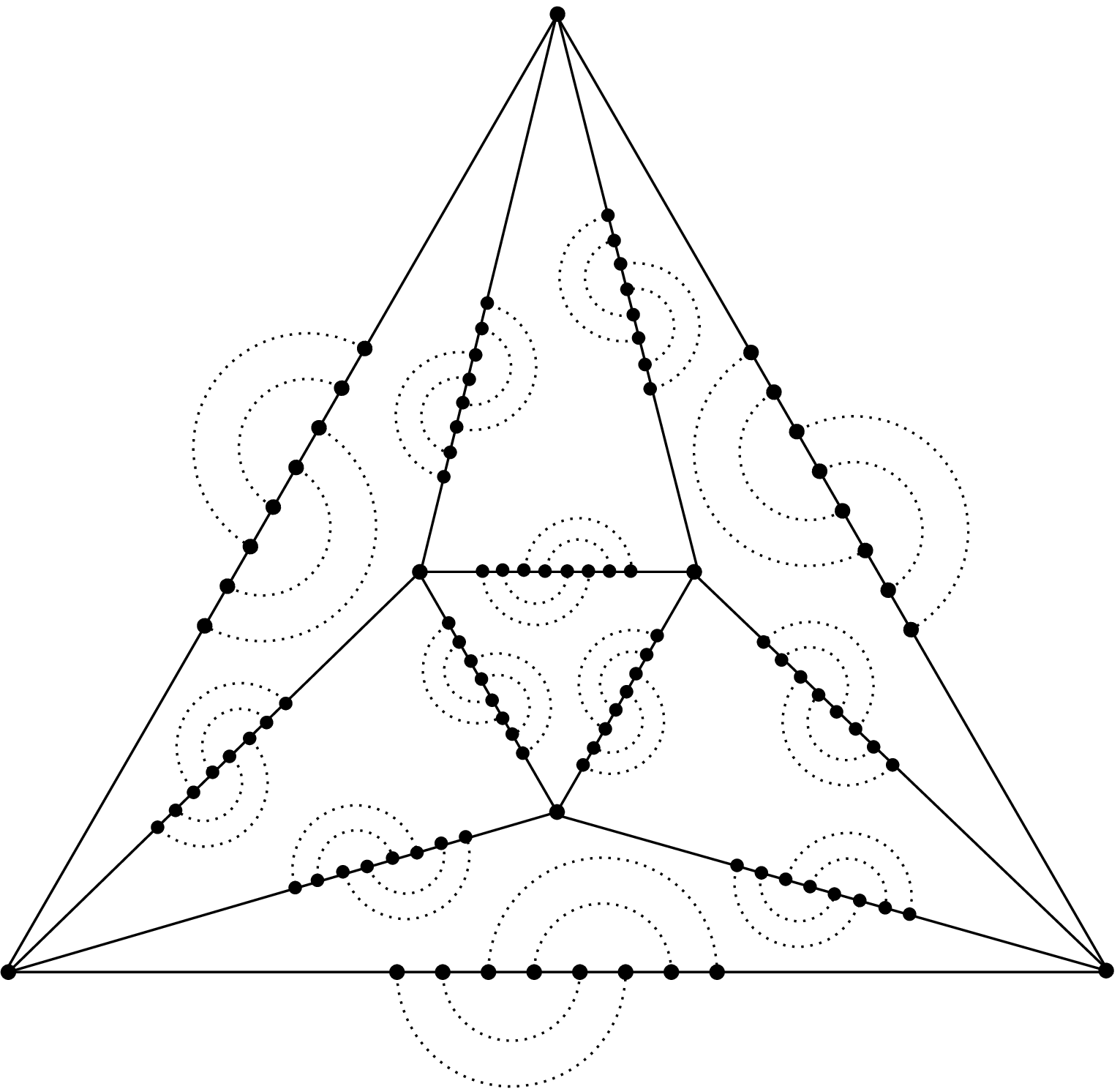}}
  \put(2.2,0){(a)}
  \put(4.7,0.6){\includegraphics[scale=0.27]{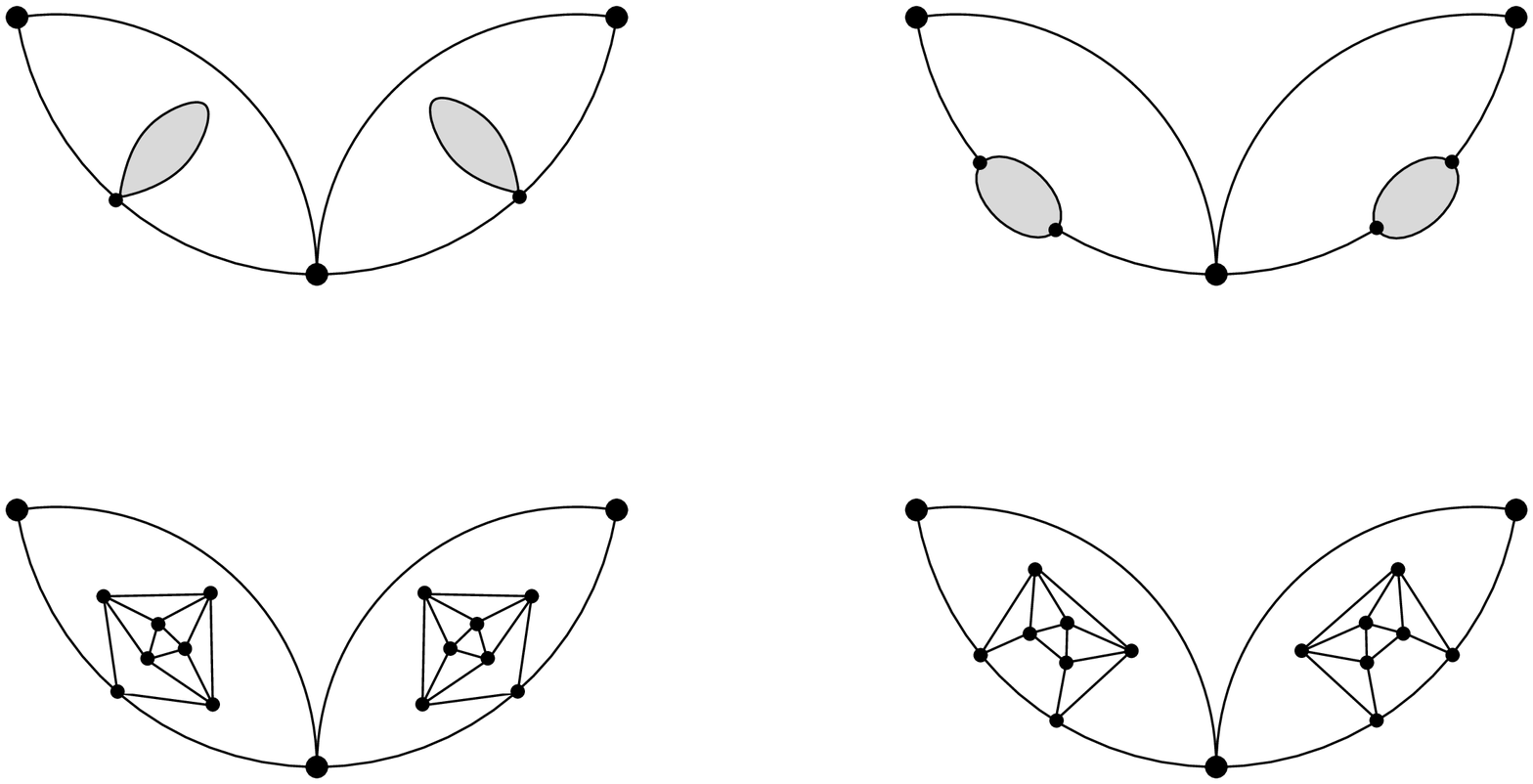}}
  \put(5.75,1.85){(b)}
  \put(5.75,0){(c)}
  \put(5.85,2.35){\scriptsize{$w$}}
  \put(4.38,3.65){\scriptsize{$v_1$}}
  \put(4.85,2.7){\scriptsize{$w_1$}}
  \put(6.8,2.7){\scriptsize{$w_2$}}
  \put(7.25,3.65){\scriptsize{$v_2$}}
  \put(9.35,1.85){(d)}
  \put(9.45,2.35){\scriptsize{$w$}}
  \put(7.98,3.65){\scriptsize{$v_1$}}
  \put(8.2,2.9){\scriptsize{$w_1$}}
  \put(8.65,2.5){\scriptsize{$w_1'$}}
  \put(10.15,2.5){\scriptsize{$w_2'$}}
  \put(10.6,2.9){\scriptsize{$w_2$}}
  \put(10.85,3.65){\scriptsize{$v_2$}}
  \put(9.35,0){(e)}
\end{picture}}
\caption{(a) The base octahedron, with dotted lines indicating placement of gadgets. Gadget-subgraphs with (b) abstract and (c) octahedral mini-gadgets are illustrated, as well as bigadget subgraphs with (d) abstract and (e) octahedral mini-bigadgets.}
\label{fig:bigcounter}
\end{figure}

The graphs so constructed all have at least 822 vertices, so the authors asked for the smallest counterexample to Lov\'asz' conjecture. As a step toward answering this question, we construct two counterexamples on 68 vertices, again one of which is 2-connected. The structure of our counterexamples is based on the configuration along one edge of Bekos and Raftopoulou's base octahedron, and in particular will use the same gadgets and mini-gadgets. The key tools that we bring are the use of M\"obius transformations (see \cite{needham} for some geometric intuition), and to work through multigraphs.

\section{A base multigraph}
Studying circle representations of multigraphs is actually somewhat easier than their simple counterparts because there are very few ways in which loops and digons can be represented. Capitalising on this, we construct the multigraph $M$ shown in Figure~\ref{fig:multicounter}(a) which consists of an 8-cycle $v_1v_2\ldots v_8v_1$ together with four pairs of neighbouring digons attached at each of $(v_1,v_4)$, $(v_2,v_7)$, $(v_3, v_6)$ and $(v_5,v_8)$. We claim that $M$ is not circle representable. 

The idea of the proof is to first show that any circle representation of $M$ must have four pairs of touching circles, this being the only way to realise the neighbouring digons, together with one additional circle representing the $8$-cycle with adjacencies as shown in Figure~\ref{fig:multicounter}(b). It then suffices to show that this configuration cannot be realised. We will handle this geometric aspect first. 

To simplify the argument, we consider an equivalent configuration obtained by applying a M\"obius transformation. Explicitly, this will be given by the inverse stereographic projection from the plane to $S^2$, followed by a rotation of the sphere that takes a chosen point to the north pole, and then a stereographic projection back to the plane. By carefully choosing a point on the circle representing the 8-cycle to be sent to infinity, the transformation maps this circle to a line and produces the following \emph{induced configuration}. Let us assume that the line in the image is the $x$-axis. The configuration then consists of eight circles $\{C_i\}_{i=1,\ldots, 8}$ where each $C_i$ has radius $r_i>0$ and touches the $x$-axis at $(t_i,0)$, numbered so that $t_1 < t_2 < \cdots < t_8$, the circles in $\{C_1,C_4,C_5,C_8\}$ are disjoint from those in $\{C_2,C_3,C_6,C_7\}$, and in addition the pairs $(C_1,C_4)$, $(C_2,C_7)$, $(C_3,C_6)$ and $(C_5,C_8)$ are touching circles. This is illustrated in Figure~\ref{fig:multicounter}(c), where the circle labelled $i$ corresponds to $C_i$.

\begin{figure}
	\small
	\centering
	\unitlength=1cm
	 \scalebox{1}{
\begin{picture}(11.7,2.6)(0,0)
  \put(0,0.5){\includegraphics[scale=0.3]{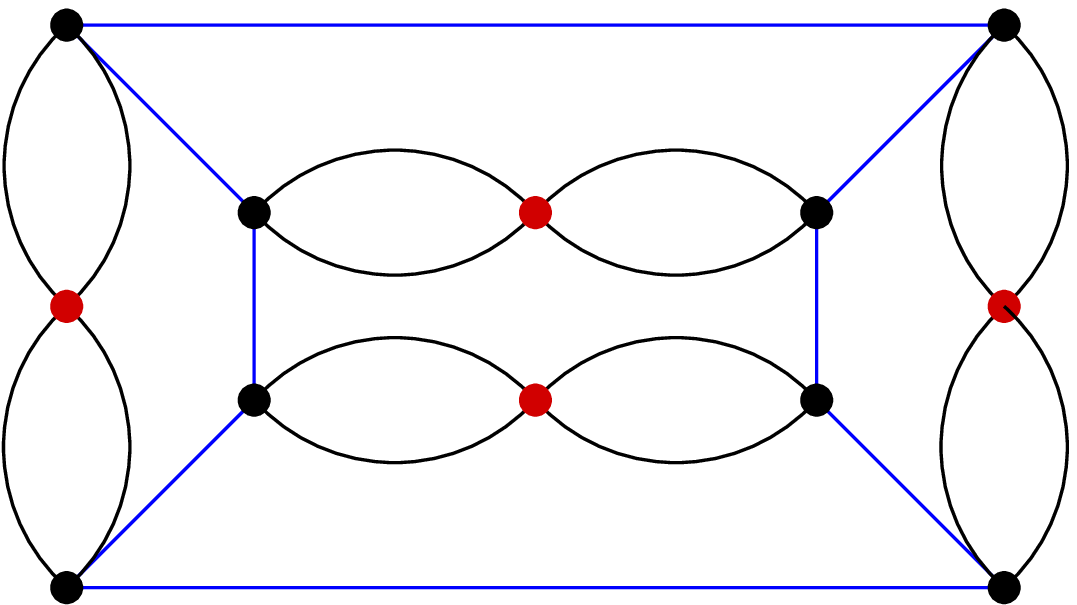}}
  \put(1.4,-0.1){(a)}
  \put(-0.15,2.35){$v_1$}
  \put(0.65,1.9){$v_2$}
  \put(0.65,0.8){$v_3$}
  \put(-0.15,0.3){$v_4$}
  \put(3.1,0.3){$v_5$}
  \put(2.3,0.8){$v_6$}
  \put(2.3,1.9){$v_7$}
  \put(3.1,2.35){$v_8$}
  \put(4.1,0.3){\includegraphics[scale=0.25]{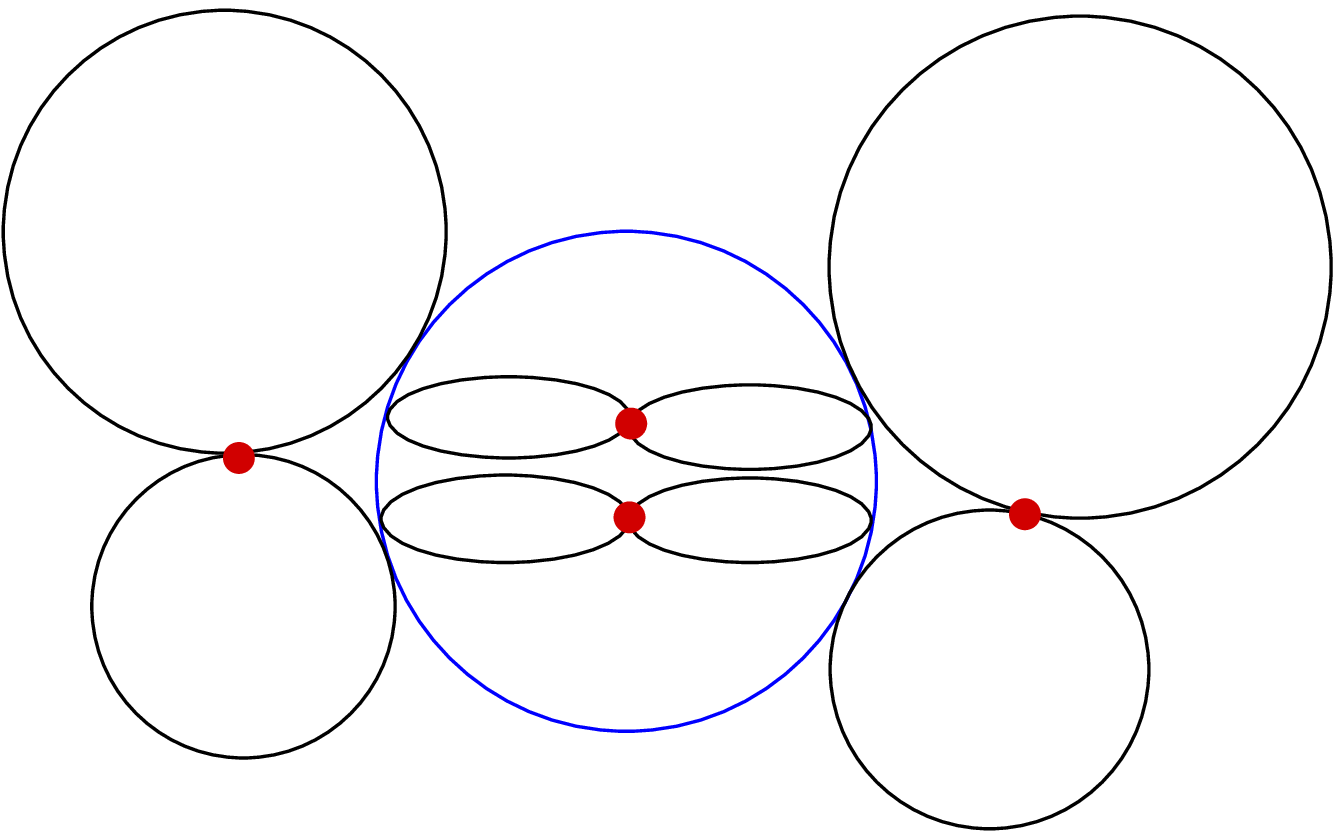}}
  \put(5.4,-0.1){(b)}
  \put(4.58,1.7){$1$}
  \put(5.32,1.28){\tiny{$2$}}
  \put(5.32,1.03){\tiny{$3$}}
  \put(4.63,0.75){$4$}
  \put(6.53,0.6){$5$}
  \put(5.94,1.03){\tiny{$6$}}
  \put(5.94,1.26){\tiny{$7$}}
  \put(6.75,1.63){$8$}
    \put(8.1,0.1){\includegraphics[scale=0.34]{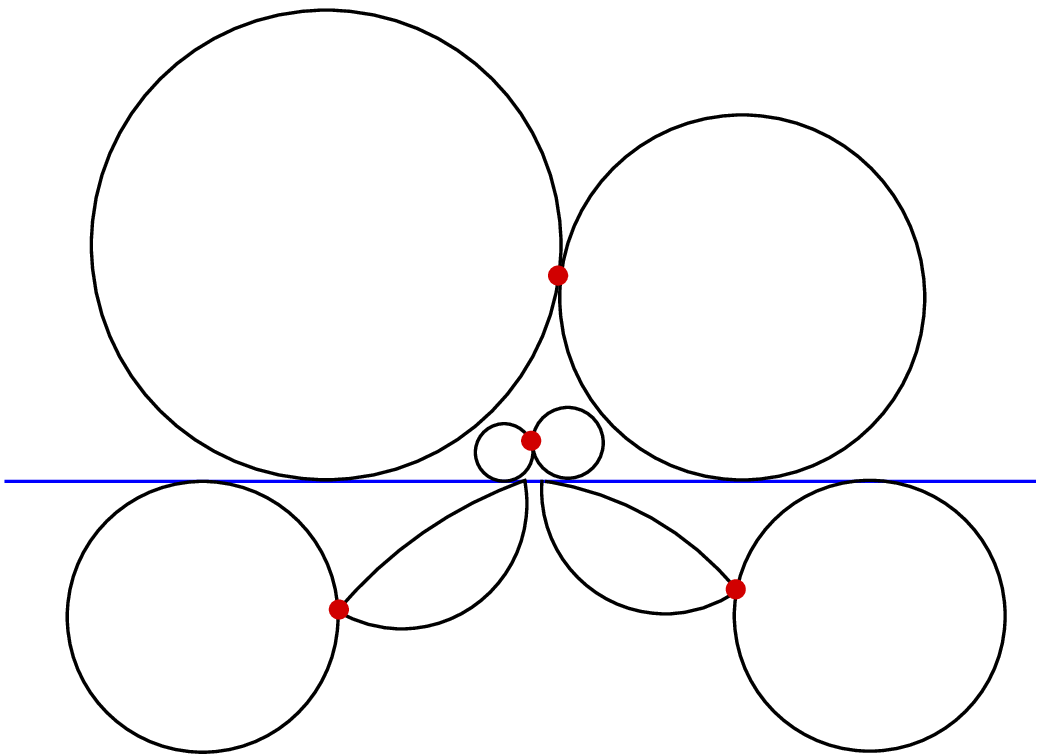}}
    \put(9.6,-0.1){(c)}
\put(8.72,0.48){$1$}
\put(9.15,1.8){$2$}
\put(9.78,1.08){\tiny{$3$}}
\put(9.58,0.7){\tiny{$4$}}
\put(10.2,0.75){\tiny{$5$}}
\put(10.0,1.12){\tiny{$6$}}
\put(10.6, 1.55){$7$}
\put(11.03,0.48){$8$}
\end{picture}}
\caption{(a) The multigraph $M$. (b) A circle configuration that must be realisable if $M$ were circle representable. (c) The induced configuration. Vertices of $M$ that are adjacent to two digons are tracked in red, and the 8-cycle in blue. Each numbered simple closed curve should be regarded as a circle.}
\label{fig:multicounter}
\end{figure}

In fact, we can be even more restrictive. Suppose we have eight circles $\{C_i(r_i,t_i)\}_{i=1,\ldots,8}$ arranged as per the induced configuration with the extra condition that $(C_1,C_8)$, $(C_4,C_5)$, $(C_2,C_3)$ and $(C_6,C_7)$ are also pairs of touching circles. We shall call this the \emph{symmetric configuration}.
\begin{lemma}\label{lemma:impossible2}
The induced configuration can be realised (by circles) only if the symmetric configuration can be realised.
\end{lemma}
\begin{proof}
Let $\{C_i(r_i,t_i)\}_{i=1,\ldots,8}$ be in the induced configuration, with $C_2$, $C_3$, $C_6$, and $C_7$ above the axis. Holding $C_2$ and $C_6$ fixed, replace $C_3$ and $C_7$ with two new circles, say $C_3'$ and $C_7'$, that are both tangent to $C_2$, $C_6$ and the axis. Then $r_3' > r_3$, which implies that $t_6-t_3 = 2\sqrt{r_3r_6} < 2\sqrt{r_3'r_6}<t_6-t_3'$, and hence $t_3' < t_3<t_4$. Here, we are using the fact that $C_3$ and $C_6$ are tangent circles that both touch the same line. Also, we know that $t_2<t_3'$ by the choice of $C_3'$ being tangent to $C_2$ and $C_6$, so this replacement preserves the order of the circles. Similarly, we have $r_7' < r_7$ from which we deduce that $t_6<t_7'<t_7<t_8$. Since the inequalities are strict at each step, no circle below the axis is tangent to any circle above the axis. The adjustments below the axis are similar. By possibly replacing $C_1$ with a larger circle with the same tangencies, we may assume that $r_1 > r_5$. Then, fixing $C_1$ and $C_5$ and replacing $C_4$ and $C_8$ produces the symmetric configuration.
\end{proof}

The following lemma contains the key geometric properties satisfied by the systems of four circles above and below the axis in the symmetric configuration.
\begin{lemma}\label{lemma:magic}
Suppose we have four circles $\{C_i(r_i,t_i)\}_{i=1,2,3,4}$ in the plane with radii $r_i>0$ and which are tangent to the $x$-axis at points $(t_i,0)$ respectively, and assume the circles are numbered so that $t_1<t_2<t_3<t_4$. In addition, suppose that $C_1$ is tangent to $C_2$ and $C_4$, and $C_3$ is tangent to $C_2$ and $C_4$. Let $n=t_4-t_1$, $m= t_3-t_2$, $\ell=t_2-t_1$ and $r=t_4-t_3$ (see Figure~\ref{fig:magic}). Then: 
\vspace{-1mm}
\begin{enumerate}[label=(\roman*)]
\item $m\cdot n = \ell\cdot r$, \vspace{-1mm}
\item $m$ is determined by $\ell$ and $r$, where $m= f(\ell,r):= \frac{-(\ell+r)+ \sqrt{(\ell+r)^2+4\ell r}}{2}$ for $\ell, r > 0$,
\item $f(\ell,r)$ is increasing in both $\ell$ and $r$ for $\ell, r > 0$.
\end{enumerate}
\end{lemma}

\begin{figure}[h]
	\centering
	\scalebox{0.8}{
	\small
\includegraphics[scale=0.6]{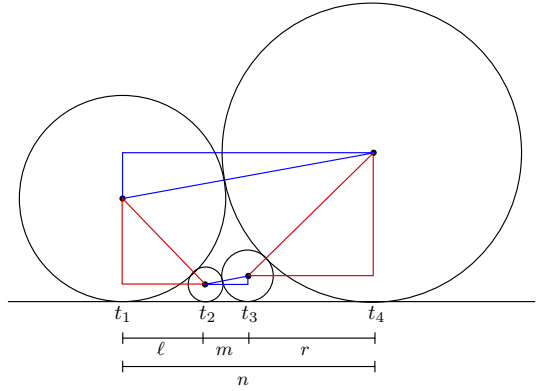}
		\put(-200,26){$t_1$}
		\put(-160,26){$t_2$}
		\put(-140,26){$t_3$}
		\put(-80,26){$t_4$}
		\put(-180,9){\footnotesize{$\ell$}}
		\put(-152,9){\footnotesize{$m$}}
		\put(-112,9){\footnotesize{$r$}}
		\put(-142,-5){\footnotesize{$n$}}
		}
\caption{Configuration of circles in Lemma \ref{lemma:magic}.}
\label{fig:magic}
\end{figure}

\begin{proof}
From Figure~\ref{fig:magic}, we observe that
\[
(t_4-t_1)^2(t_3-t_2)^2 = (4r_1r_4)(4r_2r_3) = (4r_1r_2)(4r_3r_4) = (t_2-t_1)^2(t_4-t_3)^2,
\]
which implies (i).

For (ii) we can substitute $n=\ell+m+r$ into (i) to obtain the equation $\ell r = m(\ell + m + r)$. Solving for $m$ gives the expression claimed where, since we are dealing with lengths, the only positive root of the quadratic equation has been chosen.

Using the explicit expression for $m$ as a function of $\ell$ and $r$, we can compute
\begin{align*}
\frac{\del}{\del \ell} f = \frac{1}{2}\left( \frac{\ell+3r}{\sqrt{(\ell + r)^2+4\ell r}}-1\right)
\end{align*}
which is positive whenever $\ell$ and $r$ are both positive. By symmetry, it is also true that $\frac{\del}{\del r}f$ is positive whenever $\ell$ and $r$ are both positive, which completes the proof of (iii).
\end{proof}

\begin{lemma}\label{lemma:impossible}
The symmetric and induced configurations cannot be realised.
\end{lemma}
\begin{proof}
Let $\{C_i(r_i,t_i)\}_{i=1,\ldots,8}$ be in the symmetric configuration, and note that Lemma~\ref{lemma:magic} applies to both $\{C_1,C_4,C_5,C_8\}$ and $\{C_2,C_3,C_6,C_7\}$. Let $\ell = t_3-t_2$, $m=t_6-t_3$, $r = t_7-t_6$, $\ell' = t_4-t_1$, $m'=t_5-t_4$, and $r' = t_8-t_5$. Then with $f$ as defined in Lemma~\ref{lemma:magic}(ii), we can write $m= f(\ell,r)$ and $m'=f(\ell',r')$. 

Since $t_1<t_2 < t_3 <t_4$, we have that $\ell < \ell'$. As $f$ is increasing in both variables by Lemma~\ref{lemma:magic}(iii), it follows that $m=f(\ell, r) < f(\ell', r)$. Similarly, from $t_5<t_6 < t_7 <t_8$ we obtain the inequality $r<r'$, and hence $f(\ell', r) < f(\ell', r')$ where the right hand side is now just $m'$. Altogether, this means that $m < m'$. However, we also have $t_3<t_4<t_5<t_6$ which implies that $m>m'$, giving a contradiction. This shows that the symmetric configuration cannot be realised by circles, so by Lemma~\ref{lemma:impossible2}, the induced configuration cannot be realised by circles either.
\end{proof}

\begin{theorem} \label{multicounter}
The multigraph $M$ shown in Figure~\ref{fig:multicounter}(a) is not circle representable.
\end{theorem} 
\begin{proof}
We first observe that $M$ has a unique embedding on the sphere, since the cube graph is 3-connected. In addition, given a circle representation of a graph, one can obtain circle representations with any choice of outer face by applying a M\"obius transformation that sends an interior point of that face to infinity. Thus, it is enough to show that the chosen embedding of $M$ shown in Figure~\ref{fig:multicounter}(a) does not have a circle representation.

Let $\{v_i\}_{i=1,\ldots 8}$ be the set of vertices in $M$ incident to only one digon, numbered so that $S=v_1\ldots v_8v_1$ is the cycle consisting of simple edges (see Figure~\ref{fig:multicounter}(a)). Any pair of neighbouring digons sharing exactly one vertex must be realised by two circles that touch at that common vertex. This is because if one of the digons were produced by crossing circles, then the edges of the neighbouring digon are realised by arcs of the same two circles which is only possible if the two digons share both of their vertices. Hence, $M$ must have one circle representing each digon. This means that one of the circles on which $v_i$ lies corresponds to a digon for each $i=1,2,\ldots, 8$, so the edges of $S$ incident to $v_i$ lie on the same circle. Stringing this together, we find that all of the simple edges lie on the same circle, so $S$ must be represented by a single circle, say $C_0$. Furthermore, if a pair of vertices that do not form a 2-cut are joined by exactly two parallel edges, then the parallel edges must appear consecutively in the cyclic ordering at both of those vertices. From this we conclude that each $v_i$ is a touching point.

Now suppose a circle representation of $M$ exists. It must have one circle $C_0$ corresponding to $S$, and then 8 more circles $C_1,\ldots,C_8$ labelled so that $v_i$ is the tangent point of $C_i$ with $C_0$, and these points occur in the cyclic order around $C_0$. In addition, $(C_1,C_4)$, $(C_2,C_7)$, $(C_3,C_6)$ and $(C_5,C_8)$ are pairs of touching circles. Then by applying a M\"obius transformation that sends a point on the open arc of $C_0$ representing the edge $v_1v_8$ to infinity, we would obtain a realisation of the induced configuration, thereby contradicting Lemma~\ref{lemma:impossible}. 
\end{proof}

\section{Small simple counterexamples}
From our base multigraph $M$, we proceed to construct simple counterexamples by subdividing one edge of each digon to obtain a simple graph, and then attaching gadgets at the degree 2 vertices to ensure 4-regularity. Taking eight copies of either the octahedral mini-gadget or the octahedral mini-bigadget as our gadgets gives the graphs of order 68 shown in Figure~\ref{smallcounter}. 

\begin{figure}[h]
	\small
	\centering
	\unitlength=1cm
	 \scalebox{1}{
\begin{picture}(11.9,3.2)(0,0)
  \put(0,0){\includegraphics[scale=0.2]{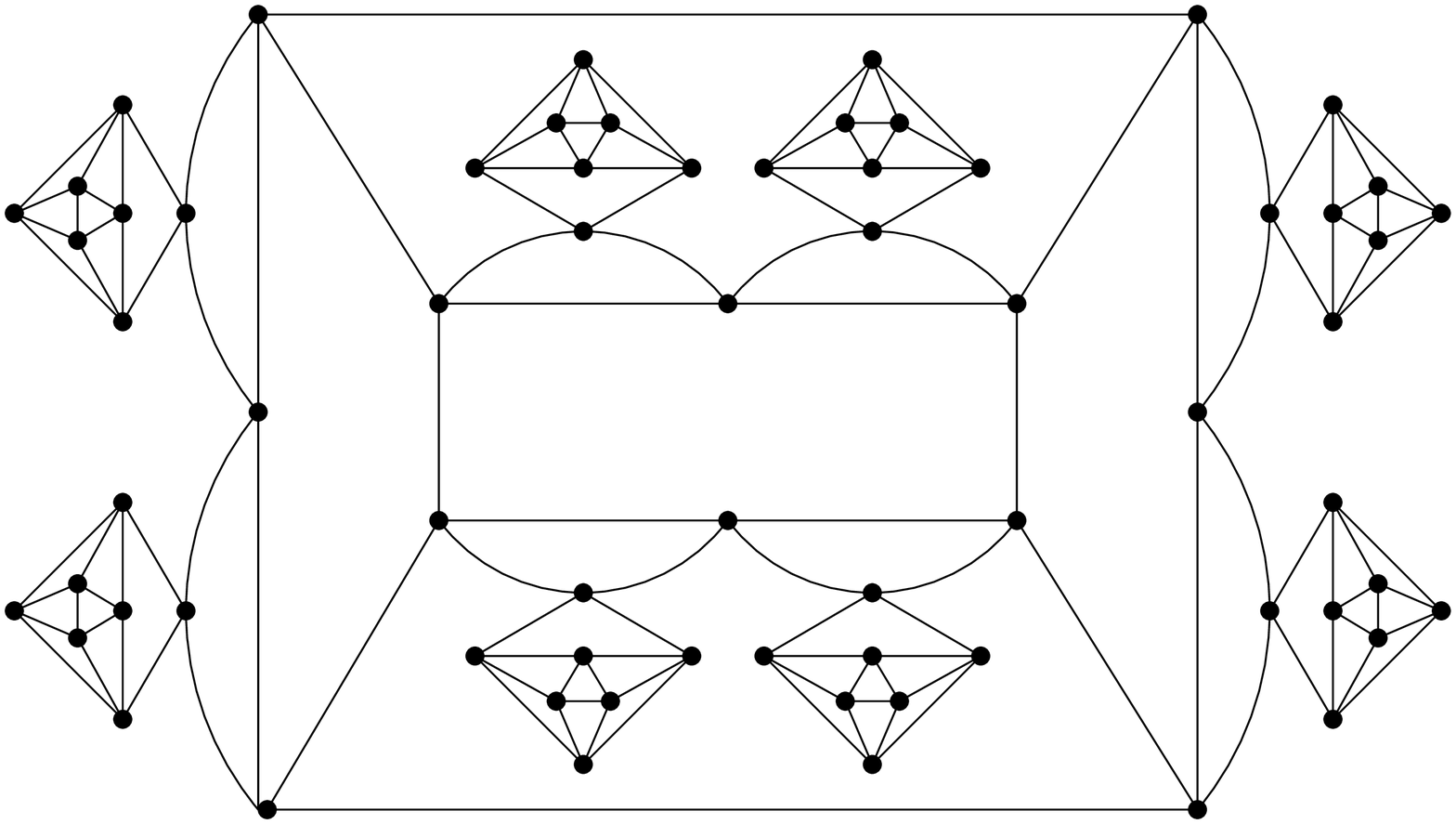}}
  \put(6.5,0){\includegraphics[scale=0.2]{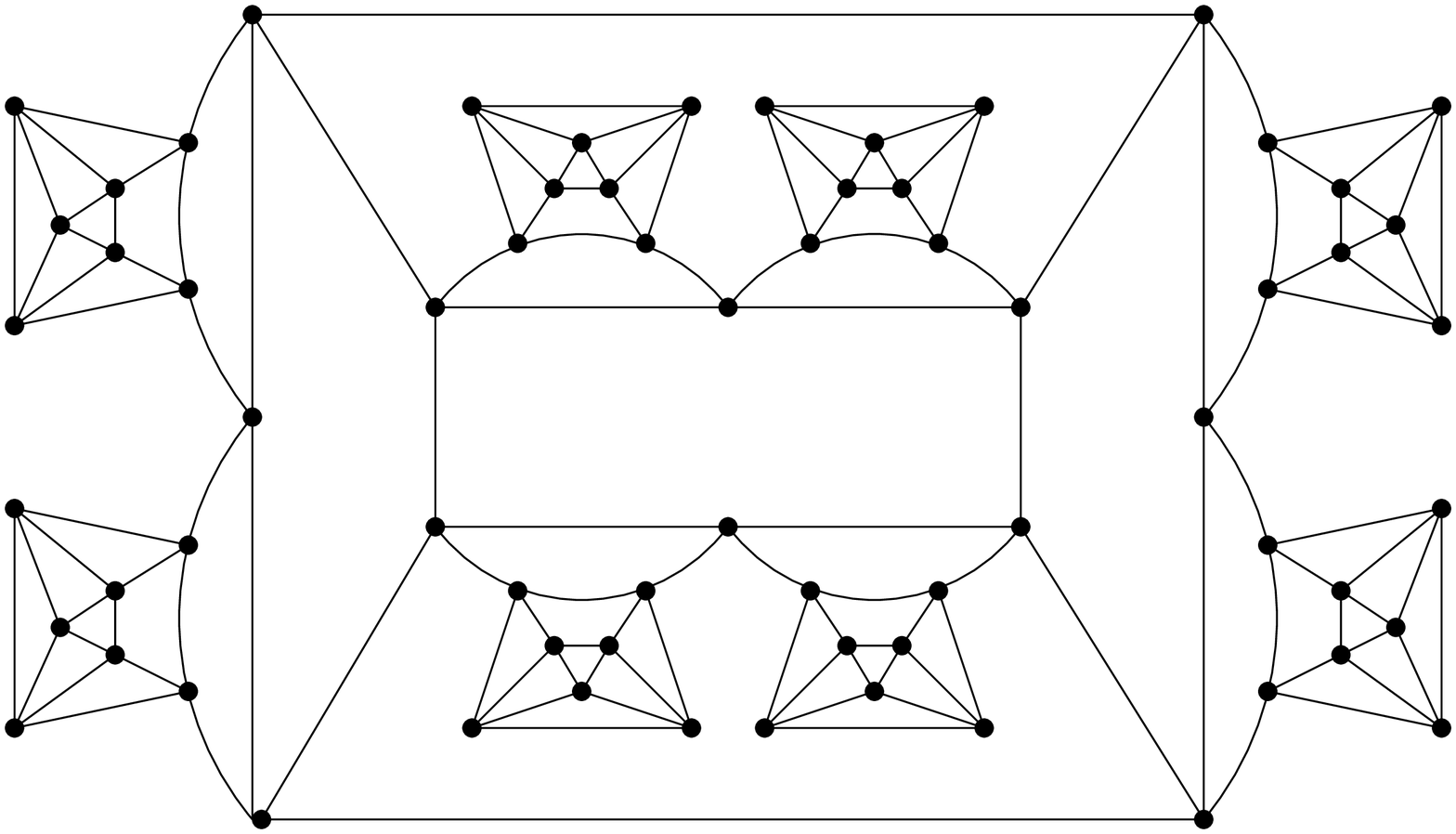}}
\end{picture}}
\caption{Simple graphs on $68$ vertices that are not circle representable.}
\label{smallcounter}
\end{figure}

\begin{lemma}[\cite{br}, Lemmas 5 and 12] \label{prune2}
Let $G$ be a 4-regular planar graph with at least one copy of the gadget-subgraph or bigadget-subgraph, labelled as in Figure~\ref{fig:bigcounter}. Then in any circle representation of $G$, each circle that contains an edge of a mini-(bi)gadget consists exclusively of edges belonging to that mini-(bi)gadget except for possibly, in the bigadget case, one circle containing the path $w_iv_iww_i'$ and another containing $w_iw_i'$ ($i=1$ or $2$).
\end{lemma}

We note that in the statement of Lemma 12 in \cite{br}, the edge denoted $(w_i,w_i')$ should read $(v_i,w)$. Our next lemma plays a similar role to Corollaries 1 and 2 in \cite{br}, the difference being that we wish to remove mini-(bi)gadgets rather than entire (bi)gadget-subgraphs. 

\begin{lemma} \label{lemma:pruning}
Let $G$ be a 4-regular planar multigraph with a specified gadget-subgraph or bigadget-subgraph. Using the labelling of Figure~\ref{fig:bigcounter}, let $G'$ be the multigraph obtained from $G$ by removing the vertices of the two mini-(bi)gadgets associated to the chosen (bi)gadget-subgraph together with all incident edges, and adding a possibly parallel edge $v_iw$ for $i=1,2$. If $G$ has a circle representation, then so does $G'$.
\end{lemma}
\begin{proof}
We shall remove one mini-(bi)gadget at a time. Suppose that the subcollection of circles containing the edges of our mini-(bi)gadget all consist exclusively of edges belonging to that mini-(bi)gadget. Then the mini-(bi)gadget can be removed by simply deleting those circles; the arcs representing its edges are gone, and at least one of the two circles whose touching or crossing point represented any given vertex has been deleted. Indeed, the only vertices for which one of its two defining circles remain are $w_i$ if we have a mini-gadget, or $w_i$ and $w_i'$ in a mini-bigadget. This leaves an arc corresponding to an extra edge $v_iw$, so we precisely have a circle representation of $G'$.

By Lemma~\ref{prune2}, the only other possibility is that among the circles representing the edges of a mini-bigadget, there are two that contain some edges belonging to the mini-bigadget and some that do not. In this case, deleting all arcs corresponding to edges of the mini-bigadget nearly gives the required circle representation of $G'$, except that the Jordan curve formed by the arcs representing $w_iv_iww_i'$ and $w_iw_i'$ is not a circle. To fix this, we transform the latter arc into that needed to complete the circle containing $w_iv_iww_i'$. This new arc existed in the original circle representation of $G$, representing some path within the now deleted mini-bigadget, so no extra intersections are produced.
\end{proof}

\begin{theorem} \label{counter}
The simple graphs shown in Figure~\ref{smallcounter} are not circle representable.
\end{theorem} 
\begin{proof}
Suppose we have a circle representation of either graph. Then applying Lemma~\ref{lemma:pruning} once for each (bi)gadget-subgraph would leave a circle representation of the base multigraph $M$ from Figure~\ref{fig:multicounter}(a). This contradicts Theorem~\ref{multicounter}. 
\end{proof}

One can generate more counterexamples on 68 vertices by using combinations of octahedral mini-gadgets and mini-bigadgets, and infinitely many larger counterexamples by attaching larger gadgets to the same base multigraph. 

 \section*{Acknowledgements}
I wish to thank Brendan McKay, Scott Morrison, and Catherine Greenhill for their invaluable guidance and generous support, as well as the anonymous reviewers for their helpful suggestions which lead to several simplifications.

\bibliography{refs_circles}
\bibliographystyle{abbrv}

\end{document}